\def\UseSection{
      \numberwithin{equation}{section}
	\theoremstyle{plain}
      \newtheorem{theorem}    {Theorem}[section]
      \DefineTheorems 
}

\def\DefineTheorems{
	
	\newtheorem{lemma}      [theorem] {Lemma}
	
	\newtheorem{prop}       [theorem] {Proposition}
	
	\newtheorem{cor}        [theorem] {Corollary}

	\theoremstyle{definition}
	\newtheorem{defn}       [theorem] {Definition}

	\theoremstyle{definition}

}

\newcommand{\bt}   {\begin{theorem}}
\newcommand{\et}   {\end  {theorem}}
\newcommand{\bl}   {\begin{lemma}}
\newcommand{\el}   {\end  {lemma}}
\newcommand{\bp}   {\begin{prop}}
\newcommand{\ep}   {\end  {prop}}
\newcommand{\bc}   {\begin{cor}}
\newcommand{\ec}   {\end  {cor}}
\newcommand{\bd}   {\begin{defn}}
\newcommand{\ed}   {\end  {defn}}

\newcommand{\ba}   {\begin{array}}
\newcommand{\ea}   {\end  {array}}
\newcommand{\be}   {\begin{enumerate}}
\newcommand{\ee}   {\end  {enumerate}}
\newcommand{\bi}   {\begin{itemize}}
\newcommand{\ei}   {\end  {itemize}}

\def\eq#1\en{\begin{equation}#1\end{equation}}  
\def\eqsplit#1\ensplit{
	\begin{equation}\begin{split}#1\end{split}\end{equation}
	}
\def\eqalign#1\enalign{
	\begin{align}#1\end{align}
	}
\def\eqmul#1\enmul{
	\begin{multline}#1\end{multline}
	}
\newcommand{\eqarrstar} {\begin{eqnarray*}} 
\newcommand{\enarrstar} {\end{eqnarray*}} 
\newcommand{\eqarray}   {\begin{eqnarray}} 
\newcommand{\enarray}   {\end{eqnarray}} 
 
\makeatletter
\newcommand{\labelcounter}[2]{{%
	\stepcounter{#1}
	\protected@write\@auxout{}%
	{\string\newlabel{#2}{{\csname the#1\endcsname}{\thepage}}}%
	{\ref{#2}}
	}}
\makeatother
\newcommand{\Zbold} {{\mathbb Z}}

\newcommand{\spose}[1] {{\hbox to 0pt{#1\hss}} }
\newcommand{\ltapprox} {\mathrel{\spose{\lower 3pt\hbox{$\mathchar"218$}}
\raise 2.0pt\hbox{$\mathchar"13C$}}}
\newcommand{\gtapprox} {\mathrel{\spose{\lower 3pt\hbox{$\mathchar"218$}}
\raise 2.0pt\hbox{$\mathchar"13E$}}}

\documentclass[12pt]{amsart}
\usepackage{enumitem}
\usepackage{amsmath,amssymb,amsthm}
\usepackage[dvips]{graphicx}
\usepackage{eepic}
\newtheorem{THM}{Theorem}[section]

\newtheorem{MOD}[THM]{Model}

\newtheorem{LEM}[THM]{Lemma}
\newtheorem{PRP}[THM]{Proposition}
\newtheorem{DEF}[THM]{Definition}

\UseSection  
\newcommand{\ra}{\rightarrow}
\newcommand{\lra}{\leftrightarrow}

\renewcommand{\to}      {\rightarrow}

\newcounter{countC}  
\newcounter{countR}  
\newcommand{\re}{\mathbb{R}}
\newcommand{\Z}{\Zbold}

\newcommand{\mc}[1]{\mathcal{#1}}
\newcommand{\mP}{\mathbb{P}}

\newcommand{\mE}{\mathbb{E}}

\newcommand{\smallW}{\scriptstyle \leftarrow}

\newcommand{\smallN}{\scriptstyle \uparrow}

\newcommand{\smallnw}{\scriptscriptstyle \nwarrow}

\newcommand{\NE}{\begin{picture}(,)
\put(2,-5){$\rightarrow$}
\put(0,.5){$\uparrow$}
\end{picture}\hspace{.5cm}
}

\newcommand{\SW}{\begin{picture}(,)
\put(0,4.8){$\leftarrow$}
\put(8.2,-0.5){$\downarrow$}
\end{picture}\hspace{.5cm}
}

\newcommand{\smallOTSP}{\begin{picture}(,)
\put(6.2,0){$\smallN$}
\put(.3,-3.3){$\smallW$}
\put(1.7,0){$\smallnw$}
\end{picture}\hspace{.4cm}
}
\newcommand{\OTSP}{\begin{picture}(,)
\put(0.5,-5){$\leftarrow$}
\put(0.5,.2){$\nwarrow$}
\put(8.7,0.5){$\uparrow$}
\end{picture}\hspace{.5cm}
}

\newcommand{\blank}[1]{}

\newcommand{\barc}{\bar{\mathcal{C}}}

\usepackage[usenames]{color}

\begin{document}

\title  {Notes on oriented percolation}

\author[Holmes]{Mark Holmes}
\address{Department of Statistics, University of Auckland}
\email{mholmes@stat.auckland.ac.nz}
\author[Salisbury]{Thomas S. Salisbury} 
\address{Department of Mathematics and Statistics, York University}
\email{salt@yorku.ca}

\subjclass[2010]{60K35}

\begin{abstract}
These notes fill in results about oriented percolation that are required for the paper \cite{HS_DREII}. Since these are essentially modifications of results found in other sources (but adapted to the model we particularly need), there is no intention to publish these. 
\end{abstract}

\maketitle

\section{Introduction}
\label{sec:intro}
This section consists of notation and results pulled from the paper  \cite{HS_DREII} that are referred to in these notes. 

For fixed $d\ge 2$, let $\mc{E}=\{\pm e_i: i=1,\dots,d\}$ be the set of unit vectors in $\Z^d$, and let $\mc{P}$
denote the power set of $\mc{E}$.
Let $\mu$ be a probability measure on $\mc{P}$.  A {\em degenerate random environment} (DRE) is a random directed graph, 
i.e.~an element $\mc{G}=\{\mc{G}_x\}_{x\in \Z^d}$ of $\mc{P}^{\Z^d}$.  We equip $\mc{P}^{\Z^d}$ with the product $\sigma$-algebra and the product measure $\mP=\mu^{\otimes \Z^d}$, so that $\{\mc{G}_x\}_{x\in \Z^d}$ are i.i.d.~under $\mP$.  We denote the expectation of a random variable $Z$ with respect to $\mP$ by $\mE[Z]$. 

We say that the DRE is {\em $2$-valued} when $\mu$ charges exactly two points, i.e.~there exist distinct $E_1,E_2\in\mc{P}$ and $p \in (0,1)$ such that $\mu(\{E_1\})=p$ and $\mu(\{E_2\})=1-p$.  As in the percolation setting, there is a natural coupling of graphs for all values of $p$ as follows.  Let $\{U_x\}_{x \in \Z^d}$ be i.i.d.~standard uniform random variables under $\mP$.  Setting
\begin{equation}
\label{eq:coupling}
\mc{G}_x=\begin{cases}
E_1, &\text{ if }U_x<p,\\
E_2, &\text{otherwise,}
\end{cases}
\end{equation}
Our principal interest is the following
\begin{MOD} $(\NE\SW)$:
\label{exa:NE_SW}
Let $E_1=\{\uparrow,\rightarrow\}$ and $E_2=\{\downarrow,\leftarrow\}$ (and set $\mu(\{E_1\})=p$, $\mu(\{E_2\})=1-p$). 
\end{MOD}
We call the generalization to $d$ dimensions the {\it orthant model} (so this is the {\it 2-$d$ orthant model}).

\begin{DEF}
\label{def:connections}
Given an environment $\mc{G}$:
\begin{itemize}
\item We say that $x$ is {\em connected} to $y$, and write $x\ra y$ if: there exists an $n\ge 0$ and a sequence $x=x_0,x_1,\dots, x_n=y$ such that $x_{i+1}-x_{i}\in \mc{G}_{x_i}$ for $i=0,\dots,n-1$. We say that $x$ and $y$ are {\em mutually connected}, or that they {\em communicate}, and write $x\lra y$ if $x\ra y$ and $y \ra x$.  
\item  Define $\mc{C}_x=\{y\in \Z^d:x \ra y\}$ (the {\em forward} cluster), $\mc{B}_y=\{x\in\Z^d:x\ra y\}$ (the {\em backward} cluster), and $\mc{M}_x=\{y \in \Z^d:x\lra y\}=\mc{B}_x\cap\mc{C}_x$ (the {\em bi-connected} cluster). \newline
Set
$\theta_+=\mP(|\mc{C}_o|=\infty)$, $\theta_-=\mP(|\mc{B}_o|=\infty)$, and $\theta=\mP(|\mc{M}_o|=\infty)$.
\item A nearest neighbour path in $\Z^d$ is {\em open in $\mc{G}$} if that path consists of directed edges in $\mc{G}$.
\end{itemize}
\end{DEF}
For model \ref{exa:NE_SW} we have
\begin{equation}
\boxed{\theta_+=1}.
\end{equation}
Simulations indicate that $\mc{C}_o$ and infinite $\mc{B}_o$ clusters have similar geometry, except that $\mc{C}_o$ typically has ``holes" whereas $\mc{B}_o$ does not.  In order to give a clearer description of this weak kind of duality, we study the geometry of
$\barc_x\supset \mc{C}_x$, defined by
\begin{multline}
\label{eq:Cbar}
\barc_x=\{z \in \Z^d: \text{ every infinite nearest-neighbour self-avoiding} \\  \text{path starting at $z$ passes through $\mc{C}_x$}\}.
\end{multline}
An important notion that arises in the proofs of these results (and elsewhere throughout this paper) is the asymptotic slope of a path.  
\begin{DEF}
\label{def:slope}
A nearest-neighbour 
path $x_0,x_1,\dots$ with $x_i=(x_i^{[1]},x_i^{[2]})\in \Z^2$ is said to have \emph{asymptotic slope $\sigma$} if 
\[\lim_{n\ra \infty}\frac{x_n^{[2]}}{x_n^{[1]}}=\sigma.\]
\end{DEF}

In \cite{HS_DREII} we state without proof a number of results about the OTSP model $(\OTSP,\cdot)$ that follow using the methods of  \cite{Dur84} for two dimensional oriented percolation models. In this model we have local environment $\mathbf{G}_x=\OTSP$ with probability $p$, and $\mathbf{G}_x=\emptyset$ with probability $1-p$, both on the triangular lattice described above. Recall that forward clusters in this model are denoted $\mathbf{C}_x$, and backward clusters $\mathbf{B}_x$. The natural coupling (\ref{eq:coupling}) gives a probability space on which the sets $\mathbf{C}_o(p)$ are increasing in $p$ almost surely, so
\[\Theta_+(p)=\mP(|\mathbf{C}_o(p)|=\infty) \quad \text{ is increasing in }p,\]
giving the critical value $p_c^{\smallOTSP}=\inf\{p:\Theta_+(p)>0\}\in (0,1)$.  

In order to describe the shape of an infinite $\mathbf{C}_x$ cluster, define $w_n=\sup\{x:(-n,x)\in \mathbf{C}_o\}$ and $v_n=\inf\{x:(-n,x)\in \mathbf{C}_o\}$.  The following Proposition is proved using subadditivity of quantities related to $w_n$. Minor modifications arise from the proofs in \cite{Dur84}, because the latter treats oriented bond percolation on the square lattice, while we need oriented site percolation on the triangular lattice.
\begin{PRP}
\label{prp:OPslope}
For the percolation model $(\OTSP,\cdot)$ with $1>p>p_c^{\smallOTSP}$, there exists $\rho=\rho_p<-1$ such that almost surely on the event $\{|\mathbf{C}_o|=\infty\}$, the upper and lower boundaries of $\mathbf{C}_o$ have asymptotic slopes $\rho$ and $1/\rho$ respectively. In other words, $\frac{w_n}{-n}\ra \rho$ and $\frac{v_n}{-n}\ra 1/\rho$ almost surely as $n\ra \infty$.  
\end{PRP}

Since $v_n$ is bounded below by a sum of independent Geometric$(1-p)$ random variables, we get the inequality $-\frac{p}{1-p}\le \rho_p$.  The following two additional Lemmas can be proved as in \cite{Dur84}.
\begin{LEM}
\label{lem:OPslope}
$\rho_p$ is continuous and strictly decreasing in $p> p_c^{\smallOTSP}$, with $\rho_p\uparrow -1$ as $p\downarrow p_c^{\smallOTSP}$.
\end{LEM}
Let $\tau=\sup\{y-x: (x,y)\in \mathbf{C}_o\}$, which measures the furthest diagonal line reached by the forward cluster of the origin.  More generally, if $z=(x_0,y_0)$, let 
$\tau_z=\sup\{(y-y_0)-(x-x_0): (x,y)\in \mathbf{C}_z\}$. Note that $|\mathbf{C}_o|=\infty\Leftrightarrow \tau=\infty$.
\begin{LEM}
\label{lem:dual_tail}
If $p>p_c^{\smallOTSP}$, then there exist constants $C$, $\gamma>0$ such that $\mP(n\le \tau<\infty)\le Ce^{-\gamma n}$.
\end{LEM}

\section{Appendix: Adapting Durrett \cite{Dur84} and Grimmett \& Hiemer \cite{GH} to OTSP.}
\label{appendixsection}

We rotate the model through $3\pi/4$ clockwise and scale by $\sqrt{2}$, so work on $\mc{L}=\{(n,m)\in \Z^2:m+n \text{ is even },n\ge 0\}$.  An occupied site at $(n,m)\in \mc{L}$ connects to each of the sites $(n+1,m+1), (n+2,m), (n+1,m-1)$. Let $\mathbf{C}_x$ be the forward cluster, i.e.~the set of points that $o$ connects to. 
Given a percolation configuration on $\mc{L}$ (all connections and points are hereafter assumed implicitly to be in $\mc{L}$) let $\xi_n=\{x:o\ra (n,x)\}$ (so $\mathbf{C}_o=\cup_{n\ge 0}\big\{\{n\}\times\xi_n^o\big\}$) and define $u_n=\sup \xi_n$ and $\ell_n=\inf \xi_n$.  Here $u_n$ and $\ell_n$ are the upper and lower boundaries of the cluster of the origin, and $o=(0,0)$ is the origin.  Set $\overline{u}_0=0=\underline{\ell}_0$, and for $n\ge 1$ let
\newcommand{\oxn}{\overline{\xi}_n}
\newcommand{\uxn}{\underline{\xi}_n}
\newcommand{\oun}{\overline{u}_n}
\newcommand{\uln}{\underline{\ell}_n}
\newcommand{\oum}{\overline{u}_m}
\newcommand{\oumn}{\overline{u}_{m,n}}
\begin{align*}
\oxn=&\{x:\text{$\exists y\le 0$ such that $(0,y)\ra (n,x)$ or $(1,y)\ra (n,x)$}\},\quad \text{ and } \quad \oun=\sup \oxn,\\
\uxn=&\{x:\text{$\exists y\ge 0$ such that $(0,y)\ra (n,x)$ or $(1,y)\ra (n,x)$}\},\quad \text{ and } \quad \uln=\inf \uxn
\end{align*}
i.e.~$\oxn$ is the set of points at level $n$ that can be reached from below the origin, and $\oun$ is the highest point at level $n$ that can be reached from below the origin (similarly from above the origin).  [To make this terminology consistent, we should connect each $(1,y)\in\mc{L}$ to $(0,y)\notin\mc{L}$, when $y\le 0$. But we will not do so.]  Note that $(1,y)\in \mc{L}$ and $y\le 0$ implies that $y\le -1$.

Then 
\begin{equation}
\xi_n=\oxn \cap [\ell_n,\infty)\text{, and on }\{\xi_n\ne \varnothing\}\text{ we have }u_n=\oun. \label{firstclaim}
\end{equation}
{\em Proof:} It is clear that $\xi_n\subset\oxn \cap [\ell_n,\infty)$. Conversely, if $(n,x)$ belongs to $\oxn \cap [\ell_n,\infty)$ then there is a lattice path from $o$ to below $(n,x)$ and a lattice path from $(0,y)$ or $(1,y)$ to $(n,x)$ (with $y\le 0$). Connect them to make piecewise linear paths. Consider their heights, when their first coordinates equal 1. In the former case this is at least $-1$, and in the latter at most $-1$. Thus the two paths cross. This must happen at a lattice point, so following first the former and then the latter gives a path from $o$ to $(n,x)$. The second statement follows likewise.\qed

\medskip

For $n>m$ let $\oumn=\sup\{x-\oum: \text{$\exists y\le \oum$ such that $(m,y)\ra (n,x)$ or  $(m+1,y)\ra (n,x)$}\}$ be the altitude gain from $\oum$ to the highest point at level $n$ that can be reached from below $\oum$.  The value of $\oum$ is determined only by the $\mathbf{G}_{k,z}$ with $k<m$ (this is true even if $m=1$, with $\oum=\pm 1$ depending on whether $o$ is open or closed, and since a site is always connected to itself).  These are independent of the $\mathbf{G}_{k,z}$ with $k\ge m$, so it follows that $\{\overline{u}_{m,m+n}:0 < n\}\sim \{\oun: 0 <n\}$ (where $\sim$ denotes equality in distribution). Likewise $\{\overline{u}_{m+1,n+1}:0 \le m<n\}\sim \{\oumn: 0 \le m<n\}$ .  Moreover 
\begin{equation}
\oum+\oumn\ge \oun\quad\text{for $n>m$.}\label{subadditivity}
\end{equation}
{\em Proof:} There is a lattice path from below $o$ to $(n,\oun)$. Connecting points gives a piecewise linear path, so let $z$ be its height when its first coordinate reaches $m$. Then either $(m,z)$ or $(m+1,z)$ is a lattice point that $\ra(n,\oun)$. In the first case, clearly $z\le\oum$. In the second, $(m-1,z)$ is an open vertex in $\overline{\xi}_{m-1}$, so also $(m,z+1)\in\overline{\xi}_{m}$, so $z<\oum$.  Therefore in either case $\oumn\ge \oun-\oum$ as required. \qed

\medskip

Recall that $\oum$ and $\oumn$ are independent. So as in \cite{Dur84}, this implies that on the event $\Omega_{\infty}=\{|{\bf C}_o|=\infty\}$ that the cluster of the origin is infinite we have that 
\eq
\frac{u_n}{n}\ra \alpha:=\inf_{n\ge 1}\mE\Big[\frac{\oun}{n}\Big]=\lim_{n\ra \infty}\frac{\oun}{n} \text{ a.s., and } \frac{\ell_n}{n}\ra -\alpha \text{ a.s.}
\label{eqn:subadd}
\en
Kingman's theorem doesn't apply here, but the conditions of Liggett {\em Ann. Probab.} (1985) do apply (in the strengthened version where the moment condition assumed is that $E[(\overline{u}_1)_+]<\infty$), and give the desired conclusion. Note that $\alpha=-\infty$ is certainly permitted within \eqref{eqn:subadd}.  The fact that $\alpha$ cannot exceed 1 is obvious since no occupied site $(n,m)$ connects to any occupied site $(n+1,m+k)$ for $k>1$ nor to any occupied site $(n+2,m+k)$ for $k>2$.  In terms of the quantities above we have 
\[\overline{u}_{n+1}\le (\overline{u}_{n}+1)\vee \overline{u}_{n-1}, \quad \underline{\ell}_{n+1}\ge (\underline{\ell}_{n}-1)\wedge \underline{\ell}_{n-1}.\]

Since $u_n\ge \ell_n$ on $\Omega_{\infty}$ we have that $\mP(\Omega_{\infty})>0\Rightarrow \alpha\ge 0$ ( since $\alpha\ge -\alpha$). Since $\mE[\overline{u}_n]$ is continuous in $p$, it follows that $\alpha$ is upper-semi-continuous in $p$ (as an infimum of continuous functions), and therefore $\alpha[p_c^{\smallOTSP}]\ge 0$. We will see below that $\alpha$ is strictly increasing (Lemma \ref{lem:strictlyincreasing}) and continuous in $p\ge p_c^{\smallOTSP}$ (Lemma \ref{lem:continuity}), and that $\alpha=0$  (Lemma \ref{lem:continuity}) and $\mP(\Omega_\infty)=0$ for $p=p_c^{\smallOTSP}$ (Lemma \ref{lem:criticalcluster}).

For $A\subset \Z$, let $\xi_n^{A}=\{x: (0,y) \ra (n,x) \text{ for some }y \in A\}$, $u^A_n=\sup\xi^A_n$, $\ell^A_n=\inf\xi^A_n$. We have the following:
\begin{equation}
\xi^A_m=\emptyset \Rightarrow \xi^A_n=\emptyset\text{ for every $n\ge m$}.\label{xiproperty}
\end{equation}
\proof Suppose that $\xi^A_m=\emptyset$.  Take $n>m$ and $(n,y)\in\xi^A_n$. There is an open lattice path from some $(0,x)$ to $(n,y)$, with $x\in A$. Joining points gives a piecewise linear path. Let $z$ be its height, when its first coordinate reaches $m$.  Then $(m,z)$ cannot be a lattice point, as then we'd have $(m,z)\in\xi^A_m$. Therefore $(m-1,z)$ is a lattice point, and is open by definition. But then $(m,z\pm 1)$ also $\in\xi^A_m$, which is a contradiction.
\qed
\medskip

Define 
\[\tau^A=\inf\{n:\xi_n^{A}=\varnothing\}.\]
By \eqref{xiproperty} this $=1+\sup\{n:\xi_{n}^{A}\ne\varnothing\}$.
Set $\tau =\tau^{\{0\}}$.  Then $\{\tau<\infty\}=\{|\mathbf{C}_o|<\infty\}$.  

\begin{LEM}
\label{lem:tau}
$\tau=\inf\{m\ge 0:\underline{\ell}_m>\overline{u}_m \}.$
\end{LEM}
\proof
Note first that $\xi_k\ne \varnothing\Rightarrow \underline{\ell}_k=\inf \xi_k\le \sup \xi_k=\overline{u}_k$.  Therefore the $\le$ part of the Lemma follows from \eqref{xiproperty}. 

Conversely, let $\tau=m\ge 1$. There are open lattice paths from $o$ to $(m-1,\underline{\ell}_{m-1})$ and from below $o$ to $(m,\overline{u}_m)$. Let $\overline{z}$ denote the height where the first coordinate of the latter path (joined up) reaches $m-1$. If $\overline{z}\ge \underline{\ell}_{m-1}$ then the two paths cross (at a lattice point), and can be spliced together as before, to show that $(m,\overline{u}_m)\in\xi_m$, contradicting $\xi_m=\emptyset$. Therefore $\overline{z}< \underline{\ell}_{m-1}$. 

One of $(m-1,\overline{z})$ or 
$(m-2,\overline{z})$ is a lattice point. If $(m-1,\overline{z})\in \mc{L}$ then $\overline{z}\le \underline{\ell}_{m-1}-2$ and $\overline{u}_m=\overline{z}+1$, so $\overline{u}_m\le \underline{\ell}_{m-1}-2+1=\underline{\ell}_{m-1}-1$. 
If $(m-2,\overline{z})$ is the lattice point then $\overline{z}\le \underline{\ell}_{m-1}-1$ and $\overline{u}_m=\overline{z}$, so again $\overline{u}_m\le \underline{\ell}_{m-1}-1$. 
A similar argument shows that $\underline{\ell}_{m}\ge \overline{u}_{m-1}+1$. Therefore 
$$
\underline{\ell}_{m}\ge\overline{u}_{m-1}+1\ge\underline{\ell}_{m-1}+1>\underline{\ell}_{m-1}-1\ge\overline{u}_{m}. \qed
$$
Let 
\begin{align*}
\overline{\xi}^z_n&=\{y: \exists x\le z\text{ such that $(0,x)\to (n,y)$ or $(1,x)\to (n,y)$}\}, \quad \overline{u}^z_n=\sup \overline{\xi}^z_n\\
\underline{\xi}^z_n&=\{y: \exists x\ge z\text{ such that $(0,x)\to (n,y)$ or $(1,x)\to (n,y)$}\}, \quad \underline{\ell}^z_n=\inf \underline{\xi}^z_n.
\end{align*}
Let $M>0$. It follows as in \eqref{firstclaim} that $\xi^{[-M,M]}_n=\overline{\xi}^M_n\cap [\ell^{[-M,M]},\infty)=\underline{\xi}^{-M}_n\cap (-\infty,u^{[-M,M]}]$, and that on $\xi^{[-M,M]}_n\neq\emptyset$ we have $u^{[-M,M]}_n=\overline{u}^M_n$ and $\ell^{[-M,M]}_n=\underline{\ell}^{-M}_n$. As in Lemma \ref{lem:tau}, 
$$
\tau^{[-M,M]}=\inf\{m\ge 0:\underline{\ell}_m^{-M}>\overline{u}_m^M \}
$$
from which we see immediately that 
\begin{equation}
\{\tau^{[-M,M]}=\infty\}\supset\{\underline{\ell}_m^{-M}\le 0\le \overline{u}_m^M \quad\forall m\}.\label{taucriterion}
\end{equation}
\begin{LEM}
\label{lem:alphabiggerthan0}
$\alpha>0\Rightarrow \mP(\Omega_\infty)>0$.
\end{LEM}
\proof If $\alpha>0$ then $\overline{u}_n\to\infty$, so we may choose $M$ such that 
$$
\mP( \overline{u}_n^M\ge 0\,\forall n)=\mP( \overline{u}_n\ge-M\,\forall n)>\frac12.
$$
Then also $\mP( \underline{\ell}_n^{-M}\le 0\,\forall n)>\frac12$. Therefore by \eqref{taucriterion}, $\mP(\Omega_\infty)>0$. \qed\bigskip

For $A\subset\mc{L}$, let $\xi_A^n=\{x:z\to (n,x)\text{ for some $z\in A$}\}$ and $u_A^n=\sup\xi_A^n$. (Note that we've transposed super/subscripts to set this apart from our earlier notation, where in any case the set  $A$ was of a different type.) 
\begin{LEM}
\label{lem:addingpoints}
Let $A\supset B$ be infinite subsets of $C=\{(i,j)\in\mc{L}: \text{$i=0$ or 1, and $j<0$}\}$. Then 
$$
\mE[u^n_{B\cup \{o\}}-u^n_B]\ge \mE[u^n_{A\cup \{o\}}-u^n_A]\ge 2p.
$$
\end{LEM}
\proof
As in \cite{Dur84}, 
$$
u^n_{B\cup \{o\}}-u^n_B=(u^n_{ \{o\}}-u^n_B)_+\ge (u^n_{ \{o\}}-u^n_A)_+=u^n_{A\cup \{o\}}-u^n_A.
$$
This shows the first inequality, and also shows that the minimal choice of $A$ is $A=C$.  The second inequality is trivially true with $n=0$ (it reads $2\ge 2p$ in this case) so suppose that $n>0$.  Then
\begin{align*}
\mE[u^n_{C\cup\{o\}}]
&=\mE[u^n_{C\cup\{o\}}1_{\text{\{$o$ open}\}}]+\mE[u^n_{C\cup\{o\}}1_{\{\text{$o$ closed}\}}]
=\mE[u^n_{C\cup\{(1,1),(2,0)\}}1_{\{\text{$o$ open}\}}]+\mE[u^n_{C}1_{\{\text{$o$ closed}\}}]\\
&=p\mE[u^n_{C\cup\{(1,1),(2,0)\}}]+(1-p)\mE[u^n_{C}]\ge p\mE[u^n_{C\cup\{o,(1,1)\}}]+(1-p)\mE[u^n_{C}]\\
&=p(\mE[u^n_{C}]+2)+(1-p)\mE[u^n_{C}]=\mE[u^n_{C}]+2p
\end{align*}
which establishes the second inequality.  To obtain the inequality we have used the fact that $o$ can only connect to $(1,-1)\in C$, $(1,1)$, and $(2,0)$. And in the second-to-last step, we use that translating $C$ by $(0,2)$ gives $C\cup\{o,(1,1)\}$. 
\qed
 \medskip
 
Recall that $\alpha[p]\ge 0$ for $p\ge p^{\smallOTSP}$ (see the discussion between \eqref{eqn:subadd} and \eqref{xiproperty}).  The following implies that $\alpha$ is strictly increasing on $[p_c^{\smallOTSP},1]$, so $\alpha>0$ for $p>p_c^{\smallOTSP}$ (a fact we will need repeatedly in what follows). 
\begin{LEM}
\label{lem:strictlyincreasing}
If $p>q$ and $\alpha[q]>-\infty$ then $\alpha[p]-\alpha[q]\ge p^2-q^2$. \end{LEM}
\proof 
Let $\alpha_n[p]=\mE[\oun[p]]$. Couple the percolation clusters for all $p$ together as usual, using uniform random variables $U_z$ at each lattice point $z\in\mc{L}$. Therefore $\overline{\xi}_n[p]\supset\overline{\xi}_n[q]$, so $\oun[p]\ge \oun[q]$. Let $\sigma$ be the first $n$ with $\oun[p]> \oun[q]$. Then $\sigma-1$ is a stopping time relative to the filtration $\mc{F}_k$ generated by the $U_{(i,j)}$ with $i\le k$.

Let $A$ be the set of lattice points $(i,x)$ with $i\ge \sigma$ that can be reached in a single step, from open  vertices $(j,y)$ satisfying $j<\sigma$ and $y\in\overline{\xi}_j[p]$. 
Let $B$ be the corresponding object, but using $\overline{\xi}_j[q]$. Then for $n\ge\sigma$ we have $\overline{\xi}_n[p]=\xi^n_A[p]$ and $\overline{\xi}_n[q]=\xi^n_{B}[q]$. Both are $\mc{F}_{\sigma-1}$-measurable. By definition, $A\supset B\cup\{(\sigma,\overline{u}_\sigma[p])\}$, and every $(x,y)\in B$ satisfies $y<\overline{u}_\sigma[p]$. Therefore by Lemma \ref{lem:addingpoints} and the strong Markov property at time $\sigma-1$,
\begin{multline*}
\mE[\oun[p]1_{\{n\ge\sigma\}}]=\mE[u^n_A[p] 1_{\{n\ge\sigma\}}]
\ge \mE[u^n_A[q] 1_{\{n\ge\sigma\}}]\\
\ge \mE\big[(u^n_{B}[q]+2q) 1_{\{n\ge\sigma\}}\big]
=\mE[\oun[q]1_{\{n\ge\sigma\}}]+2q\mP(n\ge\sigma).
\end{multline*}
Of course, $\mE[\oun[p] 1_{\{n<\sigma\}}]=\mE[\oun[q] 1_{\{n<\sigma\}}]$, so we conclude that $\alpha_n[p]\ge\alpha_n[q]+2q\mP(n\ge\sigma)$. 

At each step there is probability $p-q$ that $(k,\overline{u}_k[p])$ is open for $p$-percolation, but closed for $q$-percolation. If at least one of these events holds, for $k<n$, then $n\ge\sigma$. Therefore $\mP(n\ge\sigma)\ge 1-(1-(p-q))^n$, giving the inequality
$$
\alpha_n[p]-\alpha_n[q]\ge 2q\Big(1-(1-(p-q))^n\Big).
$$
Take $M$ large, and set $\delta=(p-q)/M$. Then
\begin{align*}
\alpha[p]-\alpha[q]&=\lim_{n\to\infty}\frac{\alpha_n[p]-\alpha_n[q]}{n}=\lim_{n\to\infty}\frac{1}{n}\sum_{k=1}^{Mn}\left[\alpha_n\Big(q+\frac{k\delta}{n}\Big)-\alpha_n\Big(q+\frac{(k-1)\delta}{n}\Big)\right]\\
&\ge \lim_{n\to\infty}\frac{1}{n}\sum_{k=1}^{Mn}2\Big(q+\frac{(k-1)\delta}{n}\Big)\Big(1-(1-\frac{\delta}{n})^n\Big)
= \frac{1-e^{-\delta}}{\delta}\int_q^{q+\delta M}2t\,dt.
\end{align*}
Sending $M\to\infty$ gives the bound $p^2-q^2$, as required.
\qed\bigskip

As in (7.2) of \cite{Dur84} we have the following result.
\begin{LEM} 
\label{lem:expldecay}
Let $\alpha'>\alpha$. 
There exist constants $C$, $\gamma>0$ depending on $\alpha,\alpha'$ such that
$$
\mP(\oun>\alpha' n)\le Ce^{-\gamma n}\quad\forall n.
$$
\end{LEM}
\proof
By \eqref{eqn:subadd}, we may find $N$ non-random such that $\mE[\frac{\overline{u}_N}{N}]<\alpha'$. Let $v_k=\overline{u}_{(k-1)N,kN}-N\alpha'$ (so $v_1=\overline{u}_N-N\alpha'$). Then the $v_k$ are IID and $\mE[v_k]<0$. Since $v_k\le N$, it follows that $\phi(\theta)=\mE[e^{\theta v_k}]<\infty$ for $\theta\ge 0$. Also let $\psi_k(\theta)=\mE[e^{\theta(\overline{u}_{k}-k\alpha' )}]$. As in \cite{Dur84}, 
$$
\limsup_{\theta\to 0}\frac{\phi(\theta)-1}{\theta}\le \mE[v_1]<0
$$
so we may find $\theta_0>0$ with $\phi(\theta_0)<1$. 
Let $n=mN+k<(m+1)N$, where $0\le k<N$. By \eqref{subadditivity}
$$
\overline{u}_{n}-\alpha'n
\le v_1+v_2+\dots+v_m+(\overline{u}_{mN,n}-\alpha'k)
$$
so 
$$
\mP(\overline{u}_{n}>\alpha'n)
\le \mE[e^{\theta_0(\overline{u}_{n}-n\alpha')}]\le \phi(\theta_0)^m\psi_k(\theta_0)=\phi(\theta_0)^{m+1}\frac{\psi_k(\theta_0)}{\phi(\theta_0)}\le \phi(\theta_0)^{n/N}\frac{\psi_k(\theta_0)}{\phi(\theta_0)}.
$$
This shows the lemma, with $C=\max_{0\le k<N}\psi_k(\theta_0)/\phi(\theta_0)$ and $\gamma=-\frac{1}{N}\log \phi(\theta_0)$. \qed

\medskip

Note that the above applies for every $p$, not just for $p>p_c^{\smallOTSP}$.\bigskip

\noindent \emph{Proof of Proposition \ref{prp:OPslope}.}
Fix $p\in (p_c,1)$.  By \eqref{eqn:subadd} there exists $\alpha[p]$ such that the model $(\OTSP,\cdot)$ rotated clockwise by $3\pi/4$ (and scaled by $\sqrt{2}$) has $n^{-1}u_n\ra \alpha$ almost surely on the event $\Omega_{\infty}=\{|{\bf C}_o|=\infty\}$.  We call this the rotated model -- it is exactly what we have been analyzing in this section. By the discussion following \ref{eqn:subadd}, $\alpha\ge 0$.  By Lemma \ref{lem:strictlyincreasing}, $\alpha[p]$ is strictly increasing for $p>p_c$, so $\alpha[p]\in (0,1)$ for all $p\in (p_c,1)$.  

Rescale by $1/\sqrt{2}$ and rotate this model, back by $\pi/4$ anticlockwise (i.e.~this is the model of \S\ref{sec:intro} rotated clockwise by $\pi/2$).  We call this the reflected model, as it can be considered as a reflection of the original model in the vertical axis.  Let $\sigma=\frac{1+\alpha}{1-\alpha}\in (1,\infty)$.  We wish to show that in the reflected model, $w_n/n \ra \sigma$ (on the event $\Omega_{\infty}$).  By symmetry, this implies that $v_n/n \ra 1/\sigma$. The result then follows immediately with $\rho=-\sigma$.

Let $\psi(u)$ denote the perpendicular projection of a point $u$ onto the diagonal (the line $y=x$), i.e.
\[\psi(u)=\Big(\frac{u^{[2]}+u^{[1]}}{2},\frac{u^{[2]}+u^{[1]}}{2}\Big).\]

Let $\epsilon\in (0,\alpha/2)$ and suppose that in the reflected model $w_n>(\sigma+\epsilon)n$.  Let $x'=(n,w_n)$. Then $z_1=\psi(x')$ is a point along the diagonal (in the reflected model) corresponding to a time $m_1=2z_1^{[1]}=x'^{[1]}+x'^{[2]}$ in the rotated model at which 
$$
u_{m_1}\ge \sqrt{2}|x'-\psi(x')|=\sqrt{2}\sqrt{\Big(\frac{n-w_n}{2}\Big)^2+\Big(\frac{n-w_n}{2}\Big)^2}=w_n-n
$$ 
Choose $\epsilon'=\epsilon'(\epsilon,\alpha)>0$ so that 
$$
\frac{1+\alpha}{1-\alpha}+\epsilon > \frac{1+\alpha+\epsilon'}{1-\alpha-\epsilon'}.
$$
Therefore $w_n>(\sigma+\epsilon)n\Rightarrow w_n(1-\alpha-\epsilon')>n(1+\alpha+\epsilon')
\Rightarrow (w_n-n)>(\alpha+\epsilon')(w_n+n)$. In other words, $u_{m_1}>(\alpha+\epsilon')m_1$. 
Since $\lim_{m\to\infty}\frac{u_m}{m}=\alpha$ almost surely on $\Omega_\infty$,
this inequality occurs for only finitely many $m_1$ almost surely (in the rotated model). 
We conclude that $w_n>(\sigma+\epsilon)n$ for only finitely many $n$, almost surely (in the reflected model).

Suppose now that $x'=(n,w_n)$ for some $w_n<(\sigma-\epsilon)n$, and let $z_1=\psi(x')$ and $m_1=n+w_n<(\sigma-\epsilon+1)n$ be as above.  Then $(n,w_n+k)\notin {\bf C}_o$ for each $k>0$.  Let $u=(n,\lfloor(\sigma+\epsilon)n\rfloor)$, and $z_2=\psi(u)$. 
In the rotated model, the point $z_2$ corresponds to a time $m_2=2z_2^{[1]}=u^{[2]}+u^{[1]}\ge(\sigma+\epsilon+1)n>(\sigma-\epsilon+1)n\ge m_1=x'^{[2]}+x'^{[1]}$. Let $\delta>0$ be a value we will choose later. Suppose also that $u_{m_2}\in ((\alpha-\delta\epsilon)m_2,(\alpha+\delta\epsilon)m_2)$ (which is true for all sufficiently large $n$).  In particular $u_{m_2}>(\alpha-\delta\epsilon)m_2$.  First consider the case that $w_n>n$. Since $(n,w_n+k)\notin {\bf C}_o$ in the reflected model, we must have a connection in the rotated model from a point $(m_1,j_1)$ or $(m_1+1,j_1)$ with $j_1\le \sqrt{2}|x'-z_1|$ to a point $(m_2,j_2)$ with $j_2\ge (\alpha-\delta\epsilon)m_2$.  This corresponds to
\[\bar{u}_{m_1,m_2}\ge j_2-j_1\ge(\alpha-\delta\epsilon)m_2-j_1\ge c_1n-j_1,\]
where $c_1=(\alpha-\delta\epsilon)(\sigma+\epsilon+1)$.
But
\[x'-z_1=\left(\frac{x'^{[1]}-x'^{[2]}}{2},\frac{x'^{[2]}-x'^{[1]}}{2}\right).\]
So 
\[j_1\le \sqrt{2}|x'-z_1|=|x'^{[1]}-x'^{[2]}|=w_n-n.\]
If, on the other hand, $w_n\le n$, we argue exactly the same way, except that the constraint on $j_1$ is simply that $j_1\le 0$. 
Since 
$0<m_2-m_1\le m_2< (\sigma+\epsilon+1)n+1\le (\sigma+\epsilon+2)n$
we have in either case that
\[\bar{u}_{m_1,m_2}>(m_2-m_1)\left[\frac{c_1n-(w_n-n)_+}{(\sigma+\epsilon+2)n}\right]
\ge (m_2-m_1)\left[\frac{c_1-\sigma+\epsilon+1}{\sigma+\epsilon+2}\right].
\]
A quick calculation shows that
$$
c_1-\sigma+\epsilon+1 = \frac{\epsilon}{1-\alpha}\Big[\alpha(1-\alpha)-2\delta-\delta\epsilon(1-\alpha)\Big]. 
$$
If we choose $\delta<\alpha(1-\alpha)/2$ and then $\epsilon$ sufficiently small, we obtain the inequality
$$
\bar{u}_{m_1,m_2}>c_2(m_2-m_1)
$$
for a constant $c_2>0$. 
Observe also that 
$$
m_2-m_1\ge (\sigma+\epsilon+1)n-(\sigma-\epsilon+1)n=2\epsilon n.
$$
We wish to apply the exponential bound of Lemma \ref{lem:expldecay}, but note that though $m_2$ is deterministic, $m_1$ is not. Nor is $\bar{u}_{m_1,m_2}$ independent of the environment looked at in order to determine $w_n$. However, if we carefully examine what we have found, it is the following: If $w_n<(\sigma-\epsilon)n$ and $u_{m_2}\in ((\alpha-\delta\epsilon)m_2,(\alpha+\delta\epsilon)m_2)$ then there is a $k$ with $m_2\ge k\ge 2\epsilon n$ such that $\bar{u}_{m_2-k,m_2}\ge c_2 k$. By Lemma \ref{lem:expldecay}, the probabilty of the former event is at most 
$$
Cm_2 e^{-2\gamma\epsilon n} \le 
C\Big((\sigma +\epsilon+1)n+1\Big) e^{-2\gamma\epsilon n}.
$$
This sums, so by Borel-Cantelli these conditions hold for only finitely many $n$, almost surely. Since $\lim_{m\to\infty}\frac{u_m}{m}=\alpha$, the event 
$u_{m_2}\in ((\alpha-\delta\epsilon)m_2,(\alpha+\delta\epsilon)m_2)$ occurs for all but finitely many $n$. Therefore in fact $w_n\ge (\sigma-\epsilon)n$ for large enough $n$, a.s.

We have proved that for all $\epsilon>0$ sufficiently small, $|n^{-1}w_n-\sigma|<\epsilon$ for all but finitely many $n$ almost surely, which establishes the result.\qed

The following result, as well as Lemmas \ref{lem:OPslope} and \ref{lem:dual_tail},  will follow from a renormalisation argument (as in \cite{Dur84}), which we turn to in section \ref{sec:OPblock}. 

\begin{LEM}
\label{lem:continuity}
$\alpha[p_c^{\smallOTSP}]=0$, and $\alpha$ is continuous on $[p_c^{\smallOTSP},1]$
\end{LEM}


\noindent\emph{Proof of Lemma \ref{lem:OPslope}.}
As in the proof of Proposition \ref{prp:OPslope}, we have $\rho=-\frac{1+\alpha}{1-\alpha}$, where $\alpha$ arises from the model studied throughout \S\ref{appendixsection}, ie a rotation and scaling of the model $(\OTSP,\cdot)$. By Lemma \ref{lem:strictlyincreasing}, $\alpha[p]$ is strictly increasing for $p>p_c$, therefore $\rho$ is strictly decreasing. By Lemma \ref{lem:continuity}, $\alpha$ is continuous in $p$, therefore so is $\rho$. The same result shows that $\alpha[p]\downarrow 0$ when $p\downarrow p_c$, which implies that $\rho\uparrow -1$. \qed

\subsection{Block construction for $p>p_c$}
\label{sec:OPblock}
The following construction is needed in order to get exponential tail decay above the critical probability. But there will be other useful consequences as well.

There is a subtle point about the construction, that isn't emphasized in \cite{Dur84}. In the latter, $\alpha$ always appears to be $=\alpha[p]$, but there are a couple of places where one can get more out by fixing $\alpha$ and letting $p$ vary. So we will treat it as a separate parameter of the construction, throughout this section, and write $\alpha[p]$ when we mean the asymptotic slope.

Fix an $\alpha>0$, and choose $\delta$ small but with $(1-\delta)\alpha$ is rational. Then choose $L$ large depending on $p$ such that $L$ is an even integer and $(1-\delta)\alpha L$ is also an even integer.  For each $(n,m)\in \mc{L}$ let 
$$
C_{n,m}=(Ln,(1-\delta)\alpha L m), \quad R_{n,m}=C_{n,m}+[0,(1+\delta)L]\times [-(1+\frac{\delta}{2})\alpha L,(1+\frac{\delta}{2})\alpha L].
$$

Let $A_{0,0}$ be the parallelogram with vertices 
\begin{align*}
w_0=&(0,-\frac{3}{2}\delta\alpha L), \quad w_1=w_0+(1+\delta)(L,\alpha L)=((1+\delta)L,(1-\frac{\delta}{2})\alpha L)\\
v_0=&(0,-\frac{1}{2}\delta\alpha L), \quad v_1=v_0+(1+\delta)(L,\alpha L)=((1+\delta)L,(1+\frac{\delta}{2})\alpha L),
\end{align*}
and let $B_{0,0}=\{(x,-y): (x,y)\in A_{0,0}\}$.  Let $H^{\nearrow}_{0,0}$ be the event that there is an open path from left to right staying in $A_{0,0}$, and similarly for $H^{\searrow}_{0,0}$. To be consistent with what we've done before, ``left'' means first coordinate 0 or 1, while ``right'' means first coordinate $(1+\delta)L$. This should probably be adjusted, because the latter might not be an integer, but we'll ignore this (as \cite{Dur84} does).  Then define
$$
G_{0,0}=H^{\nearrow}_{0,0}\cap H^{\searrow}_{0,0}, \qquad \text{ and let }G_{n,m} \text{ be }G_{0,0} \text{ translated by }C_{n,m}.
$$
In other words, $G_{n,m}=H^{\nearrow}_{n,m}\cap H^{\searrow}_{n,m}$ where $H^{\nearrow}_{n,m}$ (resp. $H^{\searrow}_{n,m}$) is the event that there is an occupied path from left to right staying in the parallelogram $A_{n,m}$ (resp. $B_{n,m}$) with vertices 
\begin{align*}
(0,\mp\frac{3}{2}\delta\alpha L)+(Ln,(1-\delta)\alpha L m)&=(Ln,\alpha L[m(1-\delta)\mp\frac32\delta]),\\
((1+\delta) L,\pm(1-\frac{\delta}{2})\alpha L)+(Ln,(1-\delta)\alpha L m)&=(L[n+1+\delta],\alpha L[m(1-\delta)\pm(1-\frac12\delta)]),\\
(0,\mp\frac{1}{2}\delta\alpha L)+(Ln,(1-\delta)\alpha L m)&=(Ln,\alpha L[m(1-\delta)\mp\frac12\delta]),\\
((1+\delta)L,\pm(1+\frac{\delta}{2})\alpha L)+(Ln,(1-\delta)\alpha L m)&=(L[n+1+\delta],\alpha L[m(1-\delta)\pm(1+\frac12\delta)]).
\end{align*}
See Figure \ref{fig:durrett_1} for a picture of the overlaps of parallelograms when $\delta=.2$, $\alpha=.75$ and $L=10$. 
\begin{figure}%
\includegraphics[scale=.5]{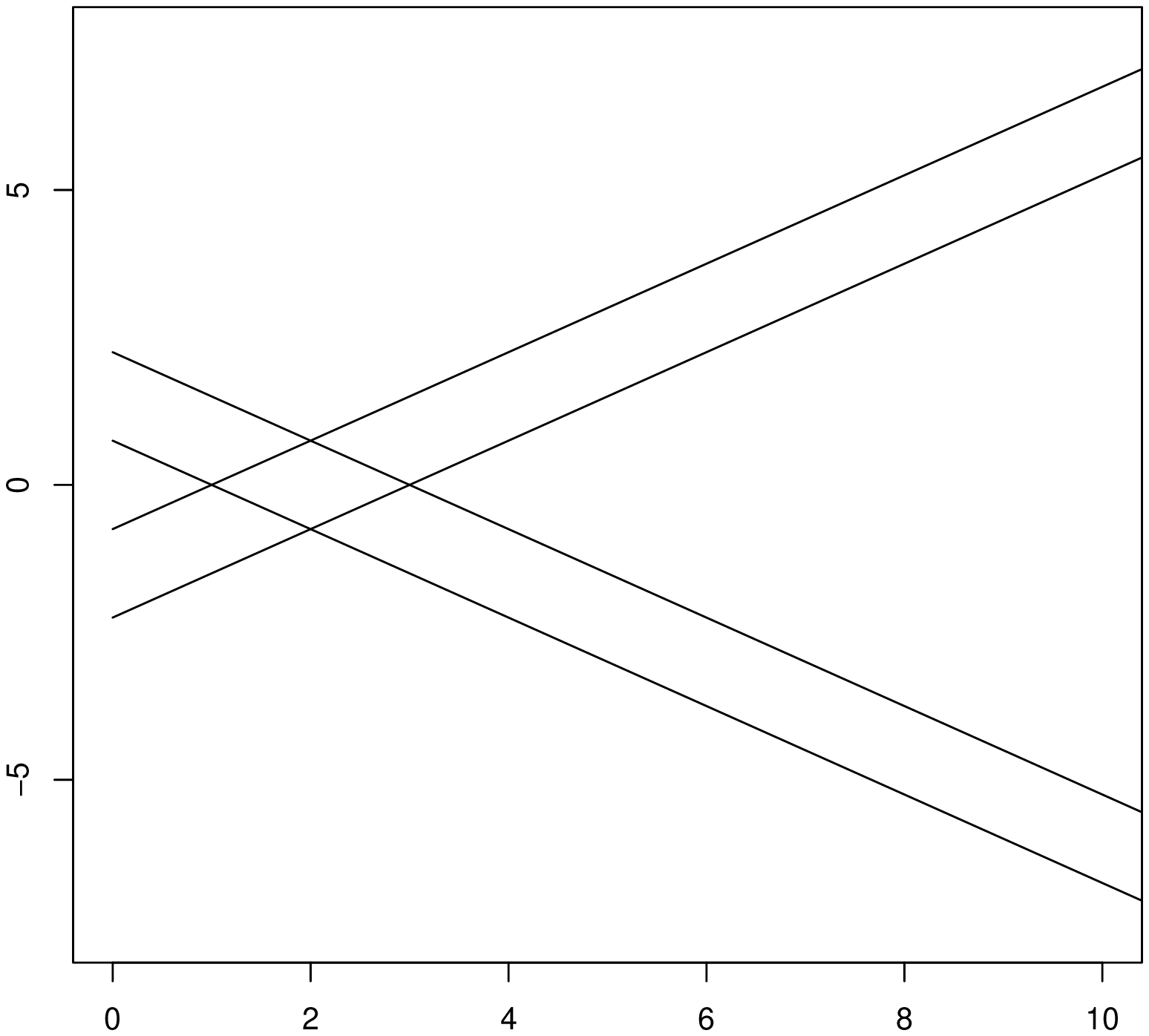}
\includegraphics[scale=.5]{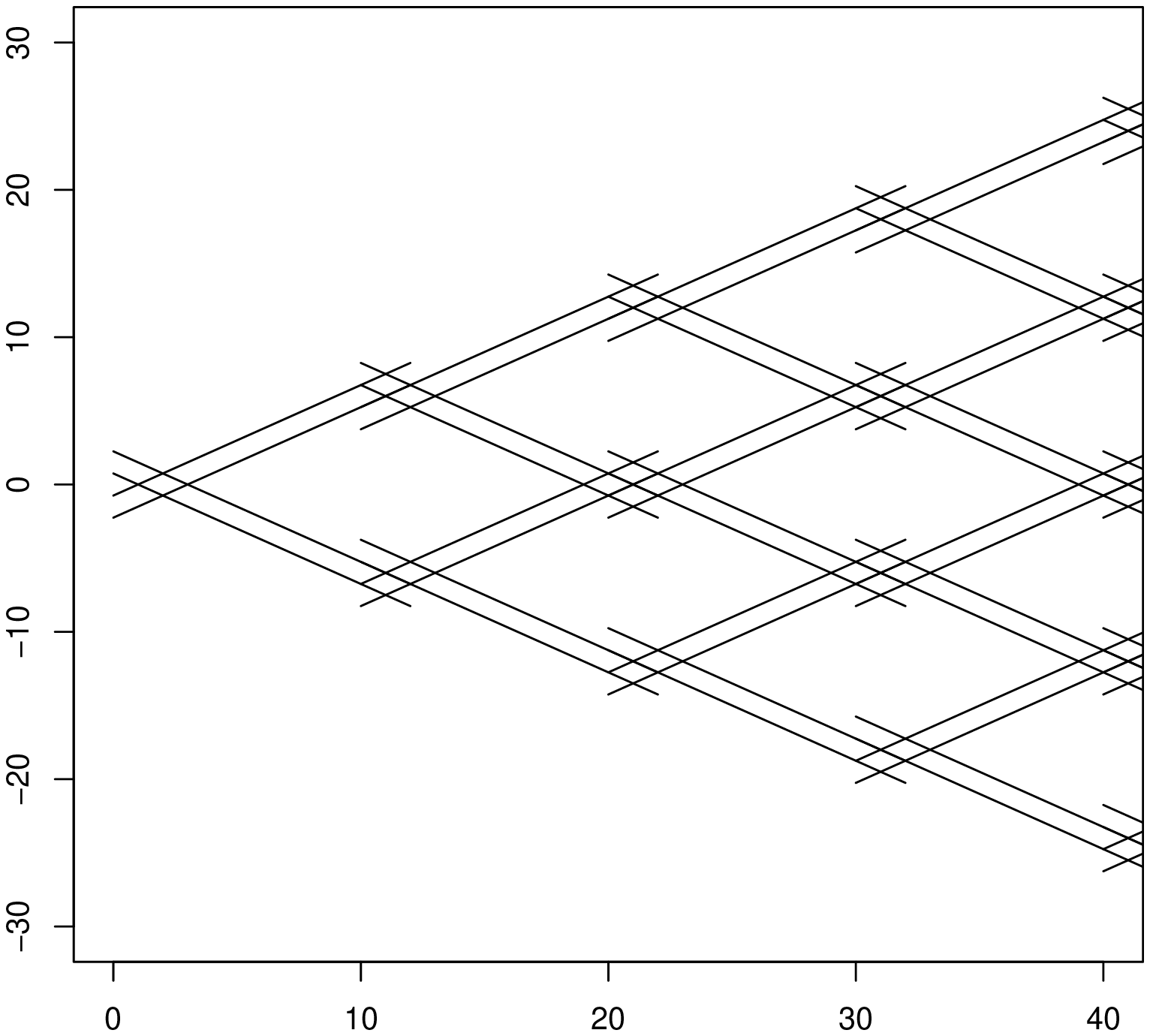}%
\caption{The construction of interlocking parallelograms for $\delta=.2$, $\alpha=.75$ and $L=10$.}%
\label{fig:durrett_1}%
\end{figure}
For example, consider the connections for the pair $(n,m)$. The entering connections are via $B_{n-1,m+1}$ (which shares a boundary segment with $B_{n,m}$) and $A_{n-1,m-1}$ (which shares a boundary segment with $A_{n,m}$).

For $z=(n,m)\in \mc{L}$ we define $\eta(z)=I_{G_{n,m}}$.  This gives a second oriented site percolation structure to $\mc{L}$, in which $z$ is {\em open} if $\eta(z)=1$. We call this the {\em  $\eta$-system}. Observe that the connections in $\mc{L}$ are now those of the square lattice rather than the triangular lattice, so percolation for the $\eta$-system means a set of open vertices which form a directed chain in the square lattice. Let $d(z,z')$ denote the graph distance between $z$ and $z'$, for the square lattice graph on $\mc{L}$. Percolation for the $\eta$-system is not IID, but we do have the following:
\begin{LEM}
\label{lem:3props}
Assume that $0<\delta<0.25$, and $0<\alpha\le 1$.
The $\eta$-system has the following properties:
\begin{itemize}
\item[(i)]  The random variables $\eta(z)$ are only $1$-dependent: if $d(z_i,z_j)>1$ for all $i,j=1, \dots, m$, $j\ne i$ then $(\eta(z_i))_{i=1}^m$ are independent. In fact, $\eta(z)$ is independent of all but 6 other vertices.
\item[(ii)] If the $\eta$-system percolates then $|\mathbf{C}_{(0,k)}|=\infty$  or $|\mathbf{C}_{(1,k)}|=\infty$ for some $k\in [-\frac{3}{2}\delta\alpha L, -\frac{1}{2}\delta\alpha L]$.
\item[(iii)] Let $\epsilon>0$, and take $p$ such that $\alpha[p]>0$. Use the $\eta$-system with $\alpha=\alpha[p]$. Then for all $L$ sufficiently large (depending on $p$), $P(\eta(z)=1)>1-\epsilon$.
\end{itemize}
\end{LEM}
\proof For (i) note that when $0<\delta<1$, $G_{n,m}$ depends only on sites in the rectangle 
\[R_{n,m}=[Ln,Ln+(1+\delta)L]\times [-(1+\frac{\delta}{2})\alpha L+(1-\delta)\alpha Lm,(1+\frac{\delta}{2})\alpha L+(1-\delta)\alpha Lm].\]
Now note that $R_{n,m}\cap R_{n+2,m}=\varnothing$ when $\delta<1$ (since then $(1+\delta)L+Ln<L(n+2)$.  Likewise $R_{n,m}\cap R_{n+1,m+3}=\varnothing$ when
\[(1+\frac{\delta}{2})\alpha L+(1-\delta)\alpha Lm<-(1+\frac{\delta}{2})\alpha L+(1-\delta)\alpha L(m+3),\]
which holds when $(2+\delta)\alpha L<3(1-\delta)\alpha L$, i.e.~when $\delta<\frac14=0.25$.  This implies that $R_{n,m}\cap R_{n,m+4}=\varnothing$, since also
\[(1+\frac{\delta}{2})\alpha L+(1-\delta)\alpha Lm<-(1+\frac{\delta}{2})\alpha L+(1-\delta)\alpha L(m+4).\]
Thus only 6 vertices in $\mc{L}$ may involve dependencies with a given vertex.
 
For (ii) note that if $G_{n,m}\cap G_{n+1,m+1}$ occurs then so does $H^{\nearrow}_{n,m}\cap H^{\nearrow}_{n+1,m+1}\cap  H^{\searrow}_{n+1,m+1}$.  By construction any occupied path fulfilling $H^{\searrow}_{n+1,m+1}$ must intersect any occupied paths fulfilling $H^{\nearrow}_{n,m}$ and $H^{\nearrow}_{n+1,m+1}$ respectively.  One might worry about the fact that in this model starting ``on the left'' allows two possible first coordinates, but because $\alpha\le 1$ this causes no problem (eg the starting points in $H^{\searrow}_{n+1,m+1}$ all lie above $H^{\nearrow}_{n,m}$. It follows that if the connected cluster of the origin in the $\eta$ system is infinite, then there is an infinite cluster in the original system starting from a both point in $0\times[w_0,v_0]$ and a point in $0\times[-v_0,-w_0]$.

We turn now to (iii), so take $\alpha=\alpha[p]>0$.  Let $z=-.8\delta\alpha L$ and recall that 
$$
\overline{u}^z_n=\sup\{y: \exists x\le z\text{ such that $(0,x)\to (n,y)$ or $(1,x)\to (n,y)$}\}.
$$
Since $\frac{\oun}{n}\ra \alpha$ as $n\ra \infty$, $\oun-n\frac{1+1.1\delta}{1+\delta}\alpha$ is eventually $<0$. Therefore we can pick $L$ large enough so that with probability at least $1-\frac{\varepsilon}{4}$, both
$$
\frac{\overline{u}_{(1+\delta)L}}{(1+\delta)L}>\frac{1+.9\delta}{1+\delta}\alpha
\quad\text{and}\quad
\oun\le .1\alpha \delta L + n\frac{1+1.1\delta}{1+\delta}\alpha\quad\forall n.
$$
Since $\{\oun\}_{n\ge 0}\sim \{\overline{u}^z_n-z\}_{n\ge 0}$, we have with probability at least $1-\frac{\varepsilon}{4}$ that both
$$
\frac{\overline{u}^z_{(1+\delta)L}+.8\delta \alpha L}{(1+\delta)L}>\frac{1+.9\delta}{1+\delta}\alpha
\quad\text{and}\quad
\overline{u}^z_{(1+\delta)L}+.8\delta \alpha L\le .1\alpha \delta L + n\frac{1+1.1\delta}{1+\delta}\alpha\quad\forall n.
$$
In other words, 
$$
\overline{u}^z_{(1+\delta)L}+.8\delta \alpha L>(1+.9\delta)\alpha L
\quad\text{and}\quad
\overline{u}^z_{n}\le -.7\alpha \delta L + n\frac{1+1.1\delta}{1+\delta}\alpha\quad\forall n.
$$
In particular $-.8\delta \alpha L+(1+.9\delta)\alpha L\le \overline{u}^z_{(1+\delta)L}\le -.7\alpha \delta L + (1+1.1\delta)\alpha L$, which simplifies to
\[(1+.1\delta)\alpha L\le \overline{u}^z_{(1+\delta)L}\le (1+.4\delta)\alpha L.\]
Thus with probability at least $1-\frac{\varepsilon}{4}$ there is an occupied path to $(1+\delta)L\times [(1+.1\delta)\alpha L,(1+.4\delta)\alpha L]$ from below $(0,-.8\delta\alpha L)$  that stays below the line $m= -.7\alpha \delta L + n\frac{1+1.1\delta}{1+\delta}\alpha$.  
This line has slope greater than $\alpha$ and therefore lies below the line $v_0\ra v_1$ since it passes through $((1+\delta)L,(1+.4\delta)\alpha L)$ and $(1+.4\delta)\alpha L<(1+.5\delta)\alpha L$.

We must consider the possibility that this path crosses the line from $w_0$ to $w_1$.  If it does then from the first crossing point, the path has to rise up to above $(1+.1\delta)\alpha L$.  So for $x\le (1+\delta)L$, let $H_x$ be the event that there is a crossing from below $(x,-.7\alpha \delta L + x\frac{1+1.1\delta}{1+\delta}\alpha)$ to above $((1+\delta)L,(1+.1\delta)\alpha L)$. We will show that for sufficiently large $L$,
\begin{equation}
P(\bigcup_{x\le (1+\delta)L}H_x)<\frac{\varepsilon}{8}.
\label{eqn:lowercrossing}
\end{equation}
It follows that for $L$ sufficiently large, $P(H^{\nearrow}_{0,0})\ge 1-\varepsilon/4-\varepsilon/8$. By symmetry also $P(H^{\searrow}_{0,0})\ge 1-\varepsilon/4-\varepsilon/8$, and (iii) then follows since $G_{0,0}=H^{\nearrow}_{0,0}\cap H^{\searrow}_{0,0}$.

To show \eqref{eqn:lowercrossing}, suppose $H_x$ occurs. 
The crossing path must have end-to-end slope at least 
\[\alpha'=\frac{\alpha L- (-\frac{3}{2}\alpha\delta L)}{(1+\delta)L}=\frac{\alpha(1+\frac{3}{2}\delta)}{1+\delta},\]
which is bigger than $\alpha$ and does not depend on $L$.
Let $M=\frac{1}{4}\delta\alpha L$, and change index to $n=(1+\delta)L-x$.  Then for all $L$ sufficiently large, Lemma \ref{lem:expldecay} implies that
\[
\sum_{0\le x\le (1+\delta)L-M}P(H_x)\le
\sum_{n=M}^{\infty}P(\oun>\alpha' n)\le \frac{\varepsilon}{8}.\]
For $x\ge (1+\delta)L-M$ in fact $P(H_x)=0$, because the slope required to reach above $(1+.1\delta)\alpha L$ is at least 
\[\frac{\alpha L- (1-\frac{\delta}{2})\alpha L}{M}=\frac{\frac{\delta}{2}\alpha L}{\frac{\delta}{4}\alpha L}=2>1,\]
which is impossible. Therefore \eqref{eqn:lowercrossing} holds, as required.
\qed

As noted following Lemma \ref{lem:strictlyincreasing}, $\alpha[p]>0$ when $p>p_c^{\smallOTSP}$, so the Lemma applies in that case.  \bigskip

Let $\mathbf{C}_o^\eta$ denote the cluster of the origin, for the $\eta$-system. 
The following is proved in \cite{Dur84}, in the version we require, but we record the proof for completeness, and to set notation. The actual result does not appear to be needed until we turn to proving Lemma \ref{lem:continuity}. It could be stated more generally for 1-dependent site percolation on the triangular lattice, rather than $\eta$-systems (though the constants would change, since 8 rather than 6 neighbours could have dependencies). The reason for assuming $\alpha>0$ is simply that the $\eta$-system doesn't make sense otherwise. 
\begin{LEM}
\label{lem:transferresult}
There exists $\varepsilon>0$ such that the following holds: For any $\delta, L, \alpha>0$ and any $p$ such that $P(\eta(z)=1)>1-\varepsilon$ we have $P(|\mathbf{C}_o^\eta|=\infty)>0$ (i.e.~the $\eta$-system percolates with positive probability).
\end{LEM}
\proof Let $D=\{(a,b)\in \re^2:|a|+|b|\le 1\}$ denote the unit diamond containing the origin.  For $N\ge 0$ let $
C_N^\eta=\cup_{n\ge 0}{}^\eta\xi_n^{[-2N,0]}$ be the set of sites connected to $0\times \{-2N,\dots,0\}$ in the $\eta$-system. Here ${}^\eta\xi_n^{[-2N,0]}$ is the set of sites $(n,x)$ reachable from below $o$ in the $\eta$-system. Note that since this system lives on $\mc{L}$ as a square lattice, these are the points reachable via square-lattice moves from some lattice point $(0,y)$ with $y\le 0$ (ie. we don't also need points $(1,y)$). We will freely adapt other earlier notation to the $\eta$-system, without necessarily spelling out all the definitions. 

Let $W=\cup_{z\in C_N^\eta}(z+D)$ denote the collection of diamonds containing the vertices of $C_N$. $W$ is a closed connected set.  If $|C_N|<\infty$ then let $\Gamma'_N$ be the boundary of the unbounded component of $W'=((-1,n)\times \re )-W$, oriented so that the line from $x=(-1,0)$ to $x'=(0,1)$ points northeast (i.e.~the boundary $\Gamma'_N$ typically turns clockwise). Let $\Gamma_N$ be the portion of $\Gamma'_N$ that starts at $(-1,0)$ and ends at $(-1,-2N)$. Let $m_1,m_2,m_3,m_4$ be the number of segments in $\Gamma_N$ of types $\nearrow,\nwarrow,\swarrow,\searrow$ respectively, so
\[m_1+m_4-m_2-m_3=0, \quad m_1+m_2-m_3-m_4=-2N.\]
There are at most $3^{m-1}$ contours of length $m$ since immediate reversals are not allowed, and contours of length $m$ satisfy
\[m_1+m_2+m_3+m_4=m.\]
As we traverse $\Gamma_N$, the unbounded component of $W'$ is always on one's left, and $W$ is always on one's right, so if $z\in C_N$, the orientation is consistent with a clockwise traverse of $z+D$. For segments in $\Gamma_N$ of type $\searrow$ or $\swarrow$, the fact that sites to the left are in $W'$ but sites to the right are in $W$ implies that sites to the right must be closed.  There are at least $(m_3+m_4)/2\ge \frac{m}{4}$ (since $m_3+m_4\ge m_1+m_2$) distinct such points (each such site can have a segment of both types associated with it).  At least $1/7$ of all of these are independent since each $\eta$-site depends on at most 6 others (translates by $(0,\pm2)$ or $(\pm1,\pm1)$).  
In other words, if we enumerate independent sites, each site in the list accounts for at most 7 sites in $C_N$. Hence there are at least $m/28$ sites determined by $\Gamma_N$ that are closed, independently of each other.
Since each such contour is at least length $2N+4$ ($4$ is the length of the boundary of a single diamond) we have 
\begin{align*}
P(\tau^{[0,2N]}_\eta<\infty)\le &\sum_{m=2N+4}^{\infty}\sum_{|\Gamma|= m}P(\Gamma_N=\Gamma)\le \sum_{m=2N+4}^{\infty}\sum_{|\Gamma|=m}P(\eta(o)=0)^{\frac{m}{28}}\\
\le &\sum_{m=2N+4}^{\infty}3^{m-1}P(\eta(o)=0)^{\frac{m}{28}}\le C(3P(\eta(o)=0)^{\frac{1}{28}})^{2N}
\end{align*}
(as long as $P(\eta(o)=0)<\frac12$, say).
Thus for all $\varepsilon>0$ sufficiently small, if $P(\eta(o)=0)<\varepsilon$ and $N$ is sufficiently large then $P(\tau^{[0,2N]}_\eta<\infty)<1$, i.e.~$P(\tau^{[0,2N]}_\eta=\infty)>0$, whence also $P(\tau_\eta=\infty)>0$.\qed

\begin{LEM}
\label{lem:explboundbelowalpha}
Suppose that $p>p_c^{\smallOTSP}$.
\begin{itemize}
\item[(1)] If $a<\alpha[p]$ then there are constants $C$ and $\gamma>0$ (depending on $a$ and $p$) such that 
\begin{equation}
\label{eqn:tailestimate}
P(\overline{u}_n\le an)\le C e^{-\gamma n}
\end{equation}
In fact, we may choose $C=1$ and $\gamma>0$ so that
\begin{equation}
\label{eqn:superadd}
\lim_{n\to\infty}\frac{1}{n}\log P(\overline{u}_n\le an)=\sup_{n\ge 1}\frac{1}{n}\log P(\overline{u}_n\le an)=-\gamma<0
\end{equation}
\item[(2)] There exist constants $C$, $\gamma>0$ such that $P(n\le \tau<\infty)\le Ce^{-\gamma n}$.
\end{itemize}
\end{LEM}
Note that we do not appear to actually use \eqref{eqn:superadd} in what follows.
\proof To prove (2) set $a=0$ in \eqref{eqn:superadd} so that $a<\alpha[p]$ (since $p>p_c^{\smallOTSP}$) and $P(\overline{u}_m\le 0)\le e^{-\gamma m}$ for all $m$.  Summing over $m\ge n$ we get
\[P(\cup_{m\ge n}\{\overline{u}_m\le 0\})\le Ce^{-\gamma n}, \quad \text{and }P(\cup_{m\ge n}\{\underline{\ell}_m\ge 0\})\le Ce^{-\gamma n}.\]
It follows that $P(\cap_{m\ge n}\{\underline{\ell}_m< 0<\overline{u}_m\})\ge 1-2Ce^{-\gamma n}$ and therefore
\[P(\cup_{m\ge n}\{\underline{\ell}_m\ge \overline{u}_m\})\le 2Ce^{-\gamma n}.\]
By Lemma \ref{lem:tau} we have
\begin{align*}
P(n\le \tau <\infty)=&P(\cup_{m=n}^{\infty}\{\tau=m\})\le P\Big(\cup_{m=n}^{\infty}\big(\{\xi_{m-1}\ne \varnothing\} \cap \{\underline{\ell}_m>\overline{u}_m\}\big)\Big)\\
\le &P\big(\cup_{m=n}^{\infty} \{\underline{\ell}_m\ge \overline{u}_m\}\big),
\end{align*}
as required.

Once we have \eqref{eqn:tailestimate}, the more refined statement 
\eqref{eqn:superadd} follows immediately. To see this, let $a_n=\log P(\oun\le an)$. By \eqref{subadditivity}, 
$$
P(\overline{u}_{m+n}\le a(m+n))\ge P(\oum\le am,\overline{u}_{m,m+n}\le an).
$$   
Taking logs, and using the independence and translation invariance properties derived earlier, we get that $a_{m+n}\ge a_m+a_n$. It follows that $\lim a_n/n$ exists and equals $\sup a_n/n$. Defining $-\gamma$ to equal the former, it follows from \eqref{eqn:tailestimate} that $-\gamma<0$, showing \eqref{eqn:superadd}.

It remains to prove \eqref{eqn:tailestimate}. We first prove it for the rescaled $\eta$-system.  Let $\overline{\chi}_n$ be the quantity corresponding to $\oxn$ and let $\overline{\chi}=\cup_n\overline{\chi}_n$, i.e.~$\overline{\chi}$ is the set of points reachable from below $(0,0)$ (all using square lattice moves, now).  Similarly $\overline{s}_n=\sup\overline{\chi}_n$ is the quantity corresponding to $\oun$.  In particular $\overline{s}_n>-\infty$ for all $n$ almost surely. 

Let $D=\{(a,b)\in \re^2:|a|+|b|\le 1\}$ denote the unit diamond containing the origin.  Let $W=\cup_{z\in \overline{\chi}}(z+D)$ denote the collection of diamonds containing the vertices of $\overline{\chi}$, and $\Gamma'_n$ be the boundary of the unbounded component of $((-1,n+1)\times \re )-W$, oriented so that the line from $x=(-1,0)$ to $x'=(0,1)$ points northeast.  Let $\Gamma_n$ be the portion of this boundary starting from $(-1,0)$ and ending at $(n,\overline{s}_n+1)$. Then $\Gamma_n$ is also part of the boundary of $W$.  For fixed $n$, let $m_1,m_2,m_3,m_4$ be the number of segments in $\Gamma_n$ of types $\nearrow,\nwarrow,\swarrow,\searrow$ respectively.  Since $\Gamma_n$ starts at $(-1,0)$ and ends at $(n,\overline{s}_n+1)$ we have 
\[m_1+m_4-m_2-m_3=n+1, \quad m_1+m_2-m_3-m_4=\overline{s}_n+1.\]
If the contour $\Gamma_n$ has length $n+1+k$ (note that $k\ge 0$) then also
\[m_1+m_2+m_3+m_4=n+1+k.\]
If also $\overline{s}_n\le q n$ for some $q<1$ then also $n+1+k-\overline{s}_n-1 \ge n+k-qn$ whence,
\[2(m_3+m_4)=m_1+m_2+m_3+m_4-\overline{s}_n-1\ge (1-q)n+k.\]
As before, all sites directly to the right of segments in $\Gamma_n$ of type $\searrow$ or $\swarrow$ must be closed.  There are at least $(m_3+m_4)/2$ distinct such points (each such site can have a segment of both types associated with it).   As before, at least 1/7 of all of these are independent since each $\eta$-site depends on at most 6 others.  Hence there are at least $(m_3+m_4)/14\ge \frac{(1-q)n+k}{28}$ sites determined by $\Gamma_n$ that are closed, independently of each other.

Now the first segment is $\nearrow$ by definition, and also by definition, no segment can be followed by a segment that reverses it (i.e.~$\nearrow$ cannot be followed by $\swarrow$ etc.) whence there are at most $3^{n+k}$ different contours with $n+k+1$ segments. We get
\begin{align*}
P(\overline{s}_n\le qn)\le &\sum_{k=0}^{\infty}P(|\Gamma_n|=n+k+1,\overline{s}_n\le qn)=\sum_{k=0}^{\infty}\sum_{\Gamma:|\Gamma|=n+k+1}P(\Gamma_n=\Gamma,\overline{s}_n\le qn)\\
\le &\sum_{k=0}^{\infty}\sum_{|\Gamma|=n+k+1}P\Big(\Gamma_n=\Gamma,\overline{s}_n\le qn,\#\text{ indep. closed sites assoc. with }\Gamma\ge \frac{(1-q)n+k}{28} \Big)\\
\le &\sum_{k=0}^{\infty}3^{n+k}(1-P(\eta(o)=0))^{\frac{(1-q)n+k}{28}}\le C\big[3(1-P(\eta(o)=0))^{\frac{1-q}{28}}\big]^n,
\end{align*}
when $P(\eta(o)=0)$ is sufficiently close to 1 (which can be achieved by taking $L$ large).  It follows that 
\begin{equation}
P(\overline{s}_n\le qn)\le Ce^{-\gamma n}  
\label{eqn:etasystembound}
\end{equation}
with $e^{-\gamma}=3(1-P(\eta(o)=0))^{\frac{1-q}{28}}$.  

We need to extend this result to the underlying model.  Let $\alpha=\alpha[p]$. If $a<\alpha$, choose $\delta< (\alpha-a)/\alpha[p]=1-a/\alpha$, so that $a<(1-\delta)\alpha$.  [This is the point at which we use the flexibility to take $\delta$ small.] We may also assume that $\delta<0.25$. Next choose $q<1$ so that $a<q(1-\delta)\alpha$.  Choose $L$ sufficiently large so that $\tilde p=P(\eta(z)=1)>1-3^{-28/(1-q)}$.  Then $3(1-\tilde p)^{(1-q)/28}<3\times 3^{-28/(1-q) \times (1-q)/28}=1$, which is the condition required for \eqref{eqn:etasystembound}.
  
Returning to the construction in Figure \ref{fig:durrett_1} and the rectangles $R_{n,m}$ we have that there is a connected set of sites in the $\eta$-system from below $0$ to the point $(n,\overline{s}_n)$, so there is also a connected set of sites in the original model from below $o$ to above $(L[n+1+\delta],\alpha L[\overline{s}_n(1-\delta)-(1+\frac{\delta}{2})])$.  Moreover, $\alpha L[\overline{s}_n(1-\delta)-(1+\frac{\delta}{2})]$ is a lower bound for the 2nd coordinates of this path, right through the box $R_{n,\overline{s}_n}$. 
It follows that if $m\in[Ln,L[n+1])$ we have
$$
\oum\ge \alpha L[\overline{s}_n(1-\delta)-(1+\frac{\delta}{2})].
$$
Therefore for $m\in[Ln,L[n+1])$,
\begin{align*}
P(\oum\le a m)\le &P(\alpha L[\overline{s}_{n}(1-\delta) -(1+\frac{\delta}{2})]\le a m)
\le P\Big(\overline{s}_{n}\le \frac{a m+\alpha L(1+\frac{\delta}{2})}{\alpha L(1-\delta)}\Big) \\
\le & P\Big(\overline{s}_{n}\le \frac{a(n+1)L+\alpha L(1+\frac{\delta}{2})}{\alpha L(1-\delta)}\Big)=P\Big(\overline{s}_{n}\le \frac{a(n+1)+\alpha (1+\frac{\delta}{2})}{\alpha(1-\delta)}\Big).
\end{align*}
For all $n$ sufficiently large, $q>\frac{a+\frac{1}{n}[a+\alpha(1+\frac{\delta}{2})]}{\alpha(1-\delta)}$.  Hence for all $n$ sufficiently large and all $m\in [Ln,L[n+1])$ we have
\[P(\oum\le a m)\le P\Big(\overline{s}_{n}\le \frac{a(n+1)+\alpha (1+\frac{\delta}{2})}{\alpha(1-\delta)}\Big)\le P(\overline{s}_{n}\le qn)\le Ce^{-\gamma n}\le C'e^{-\gamma' m}.\]
By adjusting $C'$ further, if necessary, we obtain \eqref{eqn:tailestimate} for all $m$.
\qed\bigskip

Note: A careful look at the proof shows that we don't quite need $\alpha=\alpha[p]$ in the above. For the $\eta$-system argument, what we've actually shown is that given any $q<1$, there exist $C[q]$, $\gamma[q]>0$, and $\varepsilon[q]>0$ such that \eqref{eqn:etasystembound} holds for any $\eta$-system satisfying $P(\eta(o)=0)<\varepsilon[q]$. This implies that \eqref{eqn:tailestimate} holds for any $p$ and $a$ provided we can find $\alpha$, $\delta$, $L$, and $q$ making $P(\eta(o)=0)<\varepsilon[q]$.

\begin{proof}[Proof of Lemma \ref{lem:dual_tail}]
Translating between the models, as we did in the proofs of Proposition \ref{prp:OPslope} and Lemma \ref{lem:OPslope}, we find that Lemma \ref{lem:dual_tail} is an immediate consequence of 
(2) of Lemma \ref{lem:explboundbelowalpha}.
\end{proof}

\begin{proof}[Proof of Lemma \ref{lem:continuity}]
We start with the statement that $\alpha[p_c^{\smallOTSP}]=0$. Let $p_0=p_c^{\smallOTSP}$ and $\alpha_0=\alpha[p_0]$. From upper semi-continuity of $\alpha$, we already have that $\alpha_0\ge 0$. So fix $\delta<0.25$ and choose $\varepsilon>0$ as in Lemma \ref{lem:transferresult}. Assume that $\alpha_0>0$. Then (iii) of Lemma \ref{lem:3props} applies, so we may choose $L$ so large that $P(\eta(o)=1)>1-\varepsilon$ (at $p_0$ and $\alpha_0$). We will fix this $L$ and $\alpha_0$, but allow $p$ to vary. There are only finitely many sites in $R_{0,0}$, so $P(\eta(o)=1)$ varies continuously with $p$. Therefore we can find $p<p_0$ for which $P(\eta(o)=1)>1-\varepsilon$ for this $p$ as well (and the given $L$ and $\alpha_0$). By Lemma \ref{lem:transferresult} this $\eta$-system percolates with positive probability. By (ii) of Lemma \ref{lem:3props}, it follows that the original model percolates, which is impossible when $p<p_c^{\smallOTSP}$. [Note that we are using here the ability to take $\alpha$ different from $\alpha[p]$, ie to vary $p$ but not $\alpha_0$. In particular, we couldn't do the last step for $p<p_0$ unless we could first show that $\alpha[p]>0$. The latter would be needed to show that this $\eta$-system makes sense and that the various Lemmas apply.]

Turning to continuity, the fact that $\alpha[p]$ is upper semi-continuous and increasing implies that it is right continuous on $[p_c^{\smallOTSP},1)$. So all we need show is left continuity at $p_0\in (p_c^{\smallOTSP},1]$. Set $\alpha_0=\alpha[p_0]>0$. What we must prove is that if $a<\alpha_0$ then there is an $\epsilon'>0$ such that $p_0>p>p_0-\epsilon'\Rightarrow \alpha[p]\ge a$. Given $0<a<\alpha_0$, choose $\delta$ and $q$ as in the proof of Lemma \ref{lem:explboundbelowalpha}. 
Then choose $L$ so that $P(\eta(o)=0)<\varepsilon[q]$ with $p_0$ and $\alpha_0$. (Refer to the note after the proof of Lemma \ref{lem:explboundbelowalpha} for the definition.)
Fix all these values, but allow $p$ to vary. Since the number of sites in this $R_{0,0}$ is finite, it follows that we can pick $\epsilon'>0$ such that $p_0-\epsilon'<p<p_0\Rightarrow P(\eta(o)=0)<\varepsilon[q]$ (now for $p$ and $\alpha_0$). This implies the bound \eqref{eqn:tailestimate}, for $C=C[q]$ and $\gamma=\gamma[q]$. But that inequality implies that $\alpha[p]\ge a$, as required.
\end{proof}

\subsection{Non-percolation at criticality}
The following is required in order to conclude that $\barc_o=\Z$ when $p=p_c$. It is essentially a result of Grimmett and Hiemer \cite{GH}, except that they carry it out for oriented bond percolation in a model where $(0,0)$ connects to $(1,-1)$, $(1,0)$, and $(1,1)$. This generalized earlier work of Bezuidenhout and Grimmett (\emph{Ann. Probab.} {\bf 18} (1990)), showing that the critical contact process dies out. \cite{GH} also show uniqueueness of infinite clusters, and transience of the rw on them, in the supercritical regime. 
\begin{LEM}
\label{lem:criticalcluster}
$P(\Omega_\infty)=0$ when $p=p_c^{\smallOTSP}$
\end{LEM}
\begin{proof}
The idea is to do a block construction like that described in the previous subsection. But to do it based on the assumption that $P(\Omega_\infty)>0$, rather than $\alpha>0$. This means that we are comparing the original percolation model with percolation across large blocks. The latter is essentially a 1-dependent percolation model, and Lemma \ref{lem:transferresult} applies, showing that if (for a certain $\epsilon$, the probability of percolation across a single block is $P(\eta(z)=1)>1-\epsilon$, then the 1-dependent model percolates. 

Now one must prove a result like Lemma \ref{lem:3props} that says: If $p$ has $P(\Omega_\infty)>0$ then one can choose the size of the blocks large enough to get $P(\eta(z)=1)>1-\epsilon$. But for those fixed block sizes, we are only looking at finitely many sites, so $P(\eta(z)=1)$ is continuous in $p$. Therefore it is still $>1-\epsilon$ if we move a small enough. In other words, we get it $>1-\epsilon$ for another value $p'<p$. That means the 1-dependent model still percolates, and therefore the original system percolates at this value of $p'$. Therefore  we must have $p>p_c$. In other words, the argument shows that $P(\Omega_\infty)>0\Rightarrow p>p_c$.
\end{proof}

 \section*{Acknowledgements}
Holmes's research is supported in part by the Marsden fund, administered by RSNZ. Salisbury's research is supported in part by NSERC. Both authors acknowledge the hospitality of the Fields Institute, where part of this research was conducted.

\bibliographystyle{plain}

\end{document}